\documentclass[12pt]{amsart}
\usepackage{amsmath, amsthm, amssymb, amscd}
\usepackage[headings]{fullpage}

\usepackage{xypic}
\xyoption{all}

\usepackage{tikz}
\usetikzlibrary{arrows,fit,shapes.geometric,decorations.pathmorphing}

\usepackage{mathrsfs} 

 \numberwithin{equation}{section}
 
 \theoremstyle{plain}
 \newtheorem{theorem}[equation]{Theorem}
 \newtheorem{corollary}[equation]{Corollary}
 \newtheorem{lemma}[equation]{Lemma}
 \newtheorem{proposition}[equation]{Proposition}
 \newtheorem{question}[equation]{Question}
 \newtheorem{KochenSpeckerTheorem}[equation]{Kochen-Specker Theorem}

 \theoremstyle{definition}
 \newtheorem{definition}[equation]{Definition}
 
 \newtheorem{example}[equation]{Example}

 \numberwithin{figure}{section}
 
 \DeclareMathOperator{\Spec}{Spec}
 
 \DeclareMathOperator{\partSpec}{\mathit{p}-Spec}

 \DeclareMathOperator{\Ring}{\mathsf{Ring}}
 \DeclareMathOperator{\Set}{\mathsf{Set}}
 \DeclareMathOperator{\Top}{\mathsf{Top}}
 \DeclareMathOperator{\CommRing}{\mathsf{CommRing}}

 \DeclareMathOperator{\Cstar}{\mathsf{C^*}\mathsf{Alg}}
 \DeclareMathOperator{\CommCstar}{\mathsf{Comm}\mathsf{C^*}\mathsf{Alg}}
 \DeclareMathOperator{\Gelf}{Max}
 \DeclareMathOperator{\partGelf}{\mathit{p}-Max}

 \DeclareMathOperator{\Hom}{Hom}
 \DeclareMathOperator{\Fun}{Fun}
 \DeclareMathOperator{\Obj}{Obj}

 \newcommand{\C}{\mathscr{C}}
 \newcommand{\catC}{\mathcal{C}}
 \newcommand{\Z}{\mathbb{Z}}
 \newcommand{\Complex}{\mathbb{C}}
 \newcommand{\R}{\mathbb{R}}

 \newcommand{\setS}{\mathcal{S}}
 \newcommand{\M}{\mathbb{M}}
 
 \newcommand{\Ringop}{\Ring^{\op}}
 \newcommand{\CommRingop}{\CommRing^{\op}}
 \newcommand{\Cstarop}{\Cstar^{\op}}
 \newcommand{\CommCstarop}{\CommCstar^{\op}}
 
 \newcommand{\separate}{\bigskip}
 \newcommand{\smallseparate}{\medskip}

 \newcommand{\invlim}{\varprojlim}
 \newcommand{\op}{\textnormal{op}}
 \newcommand{\restrict}{\mathfrak{r}}

 \newcommand{\comm}{\perp}
 
 \newcommand{\cat}{\mathcal}

\begin{document}

\title[Obstructing extensions of Spec]{Obstructing extensions of the functor Spec\\ to noncommutative rings}
\author{Manuel L. Reyes}
\address{Department of Mathematics\\
University of California, San Diego\\
9500 Gilman Drive \#0112\\
La Jolla, CA 92093-0112}
\email{m1reyes@math.ucsd.edu}
\urladdr{http://math.ucsd.edu/~m1reyes/}
\thanks{The author was supported by a Ford Foundation Predoctoral Diversity Fellowship at
the University of California, Berkeley, and a University of California President's Postdoctoral Fellowship at the
University of California, San Diego.
}

\date{June 2, 2011}

\begin{abstract}
This paper concerns contravariant functors from the category of rings to the category of
sets whose restriction to the full subcategory of commutative rings is isomorphic to the prime
spectrum functor $\Spec$.
The main result reveals a common characteristic of these functors: every such functor
assigns the empty set to $\M_n(\Complex)$ for $n \geq 3$.
The proof relies, in part, on the Kochen-Specker Theorem of quantum mechanics.
The analogous result for noncommutative extensions of the Gelfand spectrum functor for
$C^*$-algebras is also proved.
\end{abstract}

\maketitle

\section{Introduction}
\label{introduction section}

The prime spectrum of commutative rings and the Gelfand spectrum of commutative
$C^*$-algebras play a foundational role in the classical link between algebra and geometry,
since these spectra form the underlying point-sets of the spaces attached to a
commutative ring or $C^*$-algebra.
It is tempting to hope that one could extend these spectra to the noncommutative
setting in order to construct the ``underlying set of a noncommutative space.''
The main results of this paper (Theorems~\ref{main theorem} and~\ref{main C* theorem}
below) hinder naive attempts to do so by obstructing the existence of functors that
extend these spectra.

\smallseparate

In order to produce an obstruction, one must first fix the desired properties of the
``noncommutative spectrum'' in question.
Consider the prime spectrum $\Spec$. From the viewpoint of $\Spec$ as an underlying
point-set, two facts of key importance are (1)~the spectrum of every nonzero commutative
ring is nonempty, and (2)~the prime spectrum construction can be regarded as a
contravariant functor from the category of commutative rings to the category of sets,
\[
\Spec \colon \CommRing \to \Set.
\]
(For commutative rings, $\Spec$ is easily made into a functor because the inverse image
of a prime ideal under a ring homomorphism is again prime.)

Over the years, many different extensions of the prime spectrum to noncommutative rings have
been studied. Let $F$ be a rule assigning to each ring $R$ a set $F(R)$, such that for every
commutative ring $C$ one has $F(C) \cong \Spec(C)$. There are two desirable properties that
such an invariant may possess. 
\begin{description}
\item[Property A] \emph{For every nonzero ring $R$, the set $F(R)$ is nonempty.}
\item[Property B] \emph{The invariant $F$ can be made into a set-valued functor extending
$\Spec$,} in the sense that the assignment $R \mapsto F(R)$ is the object part of a functor
$F$ whose restriction to the category of commutative rings is isomorphic to $\Spec$.
\end{description}
Examples of invariants that satisfy Property~A include the set of prime ideals of a
noncommutative ring, Goldman's prime torsion theories~\cite{Goldman}, and the ``left spectrum''
of Rosenberg~\cite{Rosenberg1}. (These invariants satisfy Property~A because they all have
elements corresponding to maximal one- or two-sided ideals.)
Some invariants that satisfy Property~B are the spectrum of the ``abelianization'' $R \mapsto \Spec(R/[R,R])$,
the set of completely prime ideals, and the ``field spectrum'' of Cohn~\cite{Cohn}.

Each of the different ``noncommutative spectra'' listed above possess only one of the two
properties.
Our first main result states that this situation is inevitable.

\begin{theorem}
\label{main theorem}
Let $F$ be a contravariant functor from the category of rings to the category of sets whose
restriction to the full subcategory of commutative rings is isomorphic to $\Spec$. Then
$F(\M_n(\Complex)) = \varnothing$ for any $n \geq 3$.
\end{theorem}

\smallseparate

Next we state the analogous result in the context of $C^*$-algebras.
For our purposes, we define the \emph{Gelfand spectrum} of a commutative unital $C^*$-algebra
$A$ to be the set $\Gelf(A)$ of maximal ideals of $A$; these are necessarily closed in $A$.
The set $\Gelf(A)$ is in bijection with the set of \emph{characters} of $A$, which are the nonzero
multiplicative linear functionals (equivalently, unital algebra homomorphisms) $A \to \Complex$;
the correspondence associates to each character its kernel (see~\cite[Thm.~I.2.5]{Davidson}).
This is easily given the structure of a contravariant functor
\[
\Gelf \colon \CommCstar \to \Set.
\]
With appropriate topologies taken into account, the Gelfand spectrum functor provides a
contravariant equivalence between the category of commutative unital $C^*$-algebras and
the category of compact Hausdorff spaces.

The following analogue of Theorem~\ref{main theorem} provides a similar obstruction to
any noncommutative extension of the Gelfand spectrum functor.

\begin{theorem}
\label{main C* theorem}
Let $F$ be a contravariant functor from the category of unital $C^*$-algebras to the
category of sets whose restriction to the full subcategory of commutative unital
$C^*$-algebras is isomorphic to $\Gelf$. Then $F(\M_n(\Complex)) = \varnothing$ for any
$n \geq 3$.
\end{theorem}

Of course, the statements of Theorems~\ref{main theorem} and~\ref{main C* theorem} with
the category of sets replaced by the category $\Top$ of topological spaces follow as
immediate corollaries.

\smallseparate

There are plenty of results stating that a \emph{particular} spectrum of a ring or algebra is
empty.
For instance, it is easy to find examples of noncommutative $C^*$-algebras that have no
characters.
In the realm of algebra, one can think of rings that have no homomorphisms to any division
ring as having empty spectra.
For one more example, S.\,P.~Smith suggested a notion of ``closed point'' such that every
infinite dimensional simple $\Complex$-algebra has no closed points~\cite[p.~2170]{Smith}.
Notice that each of these examples assumes a fixed notion of spectrum.
The main feature setting Theorem~\ref{main theorem} apart from the arguments mentioned
above is that it applies to \emph{any} notion of spectrum satisfying Properties~A and~B
mentioned above, and similarly for Theorem~\ref{main C* theorem}.
Indeed, these spectra need not be defined in terms of ideals (either one-sided or
two-sided) or modules at all.

\smallseparate

\textbf{Outline of the proof.}
The proofs of Theorems~\ref{main theorem} and~\ref{main C* theorem} proceed roughly
as follows: (1)~construct a functor that is ``universal'' among all functors whose restriction
to the commutative subcategory is the spectrum functor; (2)~show that this functor
assigns the empty set to $\M_n(\Complex)$; (3)~by universality, conclude that
\emph{every} such functor does the same.

It is perhaps surprising that a key tool used for step~(2) above is the
\emph{Kochen-Specker Theorem}~\cite{KochenSpecker} of quantum mechanics, which forbids
the existence of certain hidden variable theories.
Recently this result has surfaced in the context of noncommutative geometry in the
\emph{Bohrification} construction introduced by C.~Heunen, N.~Landsman, and B.~Spitters
in~\cite[Thm.~6]{HeunenLandsmanSpitters}.
Those authors use the Kochen-Specker Theorem to show that a certain ``space'' associated to
the $C^*$-algebra of bounded operators on a Hilbert space of dimension $\geq 3$ has no points.
This is obviously close in spirit to Theorems~\ref{main theorem} and~\ref{main C* theorem}.
A common theme between that paper and the present one is the \emph{focus on commutative
subalgebras of a given algebra}, and we acknowledge the inspiration and influence of that work
on ours.

In the ring-theoretic case, step~(1) is achieved in Section~\ref{partial algebras section}.
The universal functor $\partSpec$ is defined in terms of \emph{prime partial ideals}, which
requires an exposition of partial algebras along with their ideals and morphisms.
Step~(2) is carried out in Section~\ref{Kochen-Specker section}, where we establish a
connection between prime partial ideals and the Kochen-Specker Theorem.
The proof of Theorem~\ref{main theorem} (basically Step~(3) above) is given in
Section~\ref{proof section}, and it is accompanied by some corollaries.
In Section~\ref{C* section} we prove Theorem~\ref{main C* theorem} by quickly
following Steps~(1)--(3) in the context of $C^*$-algebras, and we state a few of its corollaries.

\separate

\textbf{Generalizations and positive implications.}
Since the present results were announced, stronger obstructions to spectrum functors have
been proved by B.~van den Berg and C.~Heunen in~\cite{BergHeunen2}.
Those results hinder the extension of $\Spec$ and $\Gelf$ even when they are considered as
functors whose codomains are over-categories of $\Top$, such as the categories of locales
and toposes.
However, one can view these obstructions in a positive light: it seems that the actual
construction of contravariant functors extending the classical spectra necessitates a creative
choice of target category $\catC$ that contains $\Top$ (or $\Set$, if one forgets the topology).
From this perspective, the construction of ``useful'' noncommutative spectrum functors
extending the classical ones seems to remain an interesting issue.

\separate

\textbf{Conventions.}
All rings are assumed to have identity and ring homomorphisms are assumed to preserve the
identity, except where explicitly stated otherwise.
The categories of unital rings and unital commutative rings are respectively denoted by $\Ring$
and $\CommRing$.
We will consider $\Spec$ as a contravariant functor from the category of commutative rings
to the category of sets, instead of topological spaces, unless indicated otherwise.
A contravariant functor $F \colon \catC_1 \to \catC_2$ can also be viewed as a covariant functor out of the
opposite category $F \colon \catC_1^{\op} \to \catC_2$. For the most part, we will view contravariant
functors as functors that reverse the direction of arrows, in order to avoid dealing with ``opposite arrows.''
But when it is convenient we will occasionally change viewpoint and consider contravariant functors as
covariant functors out of the opposite category. 
Given a category $\catC$, we will often write $C \in \catC$ to mean that $C$ is an object of $\catC$.
When there is danger of confusion, we will write the more precise expression $C \in \Obj(\catC)$.

\section{A universal Spec functor from prime partial ideals}
\label{partial algebras section}

In this section we will define a functor $\partSpec$ that is universal among all candidates for
a ``noncommutative $\Spec$.''
We set the stage for its construction by describing the universal property that we seek.

Given categories $\catC$ and $\catC'$, we let $\Fun(\catC, \catC')$ denote the category of
(covariant) functors from $\catC$ to $\catC'$ whose morphisms are natural transformations.
(This category need not have small Hom-sets.)
The inclusion of categories $\CommRing \hookrightarrow \Ring$ induces a \emph{restriction}
functor
\begin{align*}
\restrict \colon \Fun(\Ringop, \Set) &\to \Fun(\CommRingop, \Set) \\
F &\mapsto F|_{\CommRingop},
\end{align*}
which is defined in the obvious way on morphisms (i.e., natural transformations). Now we define
the ``fiber category'' over $\Spec \in \Fun(\CommRingop, \Set)$ to be the category
$\restrict^{-1}(\Spec)$ whose objects are pairs $(F, \phi)$ with $F \in \Fun(\Ringop, \Set)$ and
$\phi \colon \restrict(F) \overset{\sim}{\longrightarrow} \Spec$ an isomorphism of functors, in
which a morphism $\psi \colon (F, \phi) \to (F', \phi')$ is a morphism $\psi \colon F \to F'$ of
functors such that $\phi' \circ \restrict(\psi) = \phi$, i.e.\ the following commutes:
\[
\xymatrix{
\restrict(F) \ar[rr]^{\restrict(\psi)} \ar[dr]_{\phi} & & \restrict(F') \ar[dl]^{\phi'} \\
& \Spec &.
}
\]
(Our use of the terminology ``fiber category'' and notation $\restrict^{-1}$ is slightly different
from other instances in the literature. The main difference is that we are considering objects that
map to $\Spec$ under $\restrict$ \emph{up to isomorphism}, rather than ``on the nose.'')

The category $\restrict^{-1}(\Spec)$ is of fundamental importance to us; we are precisely
interested in those contravariant functors from $\Ring$ to $\Set$ whose restriction to
$\CommRing$ is isomorphic to $\Spec$. The ``universal $\Spec$ functor'' $\partSpec$ that we seek
is a terminal object in this category.
The rest of this section is devoted to defining this functor and proving its universal property.

\separate

The functor $\partSpec$ to be constructed is best understood in the context of partial algebras,
whose definition we recall here. The notion of a partial algebra was defined in~\cite[\S2]{KochenSpecker}.
(A more precise term for this object would probably be \emph{partial commutative algebra,}
but we retain the historical and more concise terminology in this paper.)

\begin{definition}
A \emph{partial algebra} over a commutative ring $k$ is a set $R$ with a reflexive symmetric binary
relation $\comm \, \subseteq R \times R$ (called \emph{commeasurability}), partial addition and
multiplication operations $+$ and $\cdot$ that are functions $\comm \, \to R$, a scalar multiplication
operation $k \times R \to R$, and elements $0,1 \in A$ such that the following axioms are satisfied:
\begin{enumerate}
\item For all $a \in R$, $a \comm 0$ and $a \comm 1$;
\item The relation $\comm$ is preserved by the partial binary operations: for all $a_1, a_2, a_3 \in R$ with
$a_i \comm a_j$ ($1 \leq i,j \leq 3$) and for all $\lambda \in k$, one has $(a_1+a_2) \comm a_3$,
$(a_1 a_2) \comm a_3$, and $(\lambda a_1) \comm a_2$;
\item If $a_i \comm a_j$ for $1 \leq i,j \leq 3$, then the values of all (commutative) polynomials in
$a_1$, $a_2$, and $a_3$ form a commutative $k$-algebra.
\end{enumerate}
A \emph{partial ring} is a partial algebra over $k = \mathbb{Z}$.
\end{definition}

The third axiom of a partial algebra appears as stated in~\cite[p.~64]{KochenSpecker}. While the
axiom is succinct, it can be instructive to unravel its meaning. The third axiom is equivalent to the
following collection of axioms:
\begin{itemize}
\item[(3.0)] The element $0 \in R$ is an additive identity and $1 \in R$ is a multiplicative identity;
\item[(3.1)] Addition and multiplication are commutative when defined: if $a \comm b$ in $R$,
then $a+b = b+a$ and $ab = ba$;
\item[(3.2)] Addition and multiplication are associative on commeasurable triples: if $a \comm b$,
$a \comm c$, and $b \comm c$ in $R$, then $(a + b) + c = a + (b + c)$ and
$(a \cdot b) \cdot c = a \cdot (b \cdot c)$;
\item[(3.3)] Multiplication distributes over addition on commeasurable triples: if $a \comm b$,
$a \comm c$, and $b \comm c$ in $R$, then $a \cdot (b + c) = a \cdot b + a \cdot c$;
\item[(3.4)] Each element $a \in R$ is commeasurable to an element $-a \in R$ that is an additive
inverse to $a$ and such that $a \comm r \implies -a \comm r$ for all $r \in R$ (see the
paragraph before Lemma~\ref{eigenvalue lemma} for a discussion of uniqueness of inverses);
\item[(3.5)] Multiplication is $k$-bilinear.
\end{itemize}

\begin{definition}
A \emph{commeasurable subalgebra} of a partial $k$-algebra $R$ is a subset $C \subseteq R$
consisting of pairwise commeasurable elements that is closed under $k$-scalar multiplication
and the partial binary operations of $R$. (Thus the operations of $R$ restricted to
$C$ endow $C$ with the structure of a commutative $k$-algebra.)
\end{definition}

In particular, given any $a \in R$ one can evaluate every polynomial in $k[x]$ at $x = a$ to obtain
commeasurable $k$-subalgebra $k[a] \subseteq R$. 
More generally, any set of pairwise commeasurable elements of $R$ is contained in a commeasurable
$k$-subalgebra of $R$. Notice also that $R$ is the union of its commeasurable $k$-subalgebras.

When we need to distinguish between a $k$-algebra and a partial $k$-algebra, we shall refer to
the former as a ``full'' algebra. As the following example shows,  every full algebra can be
considered as a partial algebra in a standard way. 

\begin{example}
Let $R$ be a (full) algebra over a commutative ring $k$. We may define a relation $\comm \, \subseteq R \times R$ by
$a \comm b$ if and only if $ab = ba$ (i.e., $[a,b]=0$). This relation along with the addition, multiplication, and scalar multiplication
inherited from $R$ make $R$ into a partial algebra over $k$. For us, this is the prototypical example of a partial
algebra. We will refer to this as the ``standard partial algebra structure'' on $R$.
\end{example}

Considering a full algebra $R$ as a partial algebra is, in effect, a way to restrict our attention to
\emph{only} the commutative subalgebras of $R$. This is further amplified when one applies the
notions (defined below) of morphisms of partial algebras and partial ideals to the algebra $R$.

\begin{example}
Another important example of a partial algebra is considered in~\cite{KochenSpecker}. Let $A$ be a
unital $C^*$-algebra, and let $A_{sa}$ denote the set of self-adjoint elements of $A$. Notice that
the sum and product of two commuting self-adjoint elements is again self-adjoint, and that real scalar
multiplication preserves $A_{sa}$. 
So if $\comm \, \subseteq A_{sa} \times A_{sa}$ is the relation of commutativity (as in the previous
example), then $A_{sa}$ forms a partial algebra over $\mathbb{R}$.
\end{example}

Just as one may study ideals of a $k$-algebra, we will consider ``partial ideals'' of a partial $k$-algebra.

\begin{definition}
Let $R$ be a partial algebra over a commutative ring $k$. A subset $I \subseteq R$ is a \emph{partial ideal} of
$R$ if, for all $a, b \in R$ such that $a \comm b$, one has:
\begin{itemize}
\item $a, b \in I \implies a+b \in I$;
\item $b \in I \implies ab \in I$.
\end{itemize}
Equivalently, a partial ideal of $R$ is a subset $I \subseteq R$ such that, for every commeasurable
subalgebra $C \subseteq R$, the intersection $I \cap C$ is an ideal of $C$.
If $R$ is a (full) $k$-algebra, then a partial ideal of $R$ is a partial ideal of the standard partial
algebra structure on $R$. 
\end{definition}

To better understand the set of partial ideals of an arbitrary (full or partial) algebra, it helps to
consider some general examples.
Let $R$ be an algebra over a commutative ring $k$. If $I$ is a left, right, or two-sided ideal of $R$, then $I$ is
a clearly a partial ideal of $R$. Furthermore, when $R$ is commutative the partial ideals of $R$ are precisely the
ideals of $R$.

\begin{lemma}
\label{improper partial ideal}
Let $I$ be a partial ideal of a partial $k$-algebra $R$. Then $I = R$ if and only if $1 \in I$.
\end{lemma}

\begin{proof}
(``If'' direction.) If $1 \in I$, then $1 \comm R$ gives $R = (R \cdot 1) \subseteq I$. Hence $I = R$.
\end{proof}

\begin{proposition}
\label{division ring partial ideals}
Let $D$ be a division ring. Then the only partial ideals of $D$ are $0$ and $D$. 
\end{proposition}

\begin{proof}
Suppose that $I \subseteq D$ is a nonzero partial ideal, and let $0 \neq a \in I$. Then $a \comm a^{-1}$, so
$1 = a^{-1} \cdot a \in I$. It follows from Lemma~\ref{improper partial ideal} that $I = D$.
\end{proof}

Yet another example of a partial ideal in an arbitrary ring $R$ is the set $N \subseteq R$ of
nilpotent elements of $R$.
Indeed, for any commutative subring $C$ of $R$, $N \cap C$ is the nilradical of $C$ and hence
is an ideal of $C$.
It is well-known that the set of nilpotent elements of a ring $R$ is not even closed under addition
for many noncommutative rings $R$. In fact, it is hard to find \emph{any} structural properties
that this set possesses for a general ring $R$, making this observation noteworthy.
(This example also illustrates that ring theorists must take particular care not to impose their
usual mental images of ideals upon the notion of a partial ideal.)

\separate

We now introduce a notion of prime partial ideal, which will provide a type of ``spectrum.''

\begin{definition}
\label{prime partial ideal definition}
A partial ideal $P$ of a partial $k$-algebra $R$ is \emph{prime} if $P \neq R$ and whenever $x \comm y$
in $A$, $xy \in P$ implies that either $x \in P$ or $y \in P$.
Equivalently, a partial ideal $P$ of $R$ is prime if $P \subsetneq R$ and for every commeasurable
subalgebra $C \subseteq R$, $P \cap C$ is a prime ideal of $C$. The set of prime partial ideals of a
(full) $k$-algebra $R$ is denoted $\partSpec(R)$.
\end{definition}

If $R$ is a commutative $k$-algebra, then the prime partial ideals of $R$ are precisely the prime ideals of $R$.
Now the fact that $\Spec \colon \CommRing \to \Set$ defines a (contravariant) functor depends on the fact
prime ideals behave well under homomorphisms of commutative rings. It turns out that prime partial ideals behave
just as well, provided that one uses the ``correct'' notion of a morphism of partial algebras. This is proved in
Lemma~\ref{preimage lemma} below. The following definition was given in~\cite[\S2]{KochenSpecker}.

\begin{definition}
Let $R$ and $S$ be partial algebras over a commutative ring $k$. A \emph{morphism of partial algebras} is a
function $f \colon R \to S$ such that, for every $\lambda \in k$ and all $a, b \in R$ with $a \comm b$,
\begin{itemize}
\item $f(a) \comm f(b)$,
\item $f(\lambda a) = \lambda f(a)$,
\item $f(a + b) = f(a) + f(b)$,
\item $f(ab) = f(a) f(b)$,
\item $f(0) = 0$ and $f(1) = 1$.
\end{itemize}
(In other words, $f$ preserves the commeasurability relation and its restriction to every
commeasurable subalgebra $C \subseteq R$ is a homomorphism of commutative $k$-algebras
$f|_C \colon C \to f(C)$.)
\end{definition}

Of course, any algebra homomorphism $R \to S$ of $k$-algebras is also a morphism of partial algebras when
$R$ and $S$ are considered as partial algebras.

\begin{lemma}
\label{preimage lemma}
Let $f \colon R \to S$ be a morphism of partial $k$-algebras, and let $I$ be a partial ideal of $S$. 
\begin{enumerate}
\item The set $f^{-1}(I) \subseteq R$ is a partial ideal of $R$.
\item If $I$ is prime, then $f^{-1}(I)$ is also prime.
\end{enumerate}
This holds, in particular, when $R$ and $S$ are (full) algebras, $f$ is a $k$-algebra homomorphism,
and $I$ is a (prime) partial ideal of $S$.
\end{lemma}

\begin{proof}
Let $a, b \in R$ be such that $a \comm b$. Then $f(a) \comm f(b)$. If $a, b \in f^{-1}(I)$ then $f(a),f(b) \in I$.
Thus $f(a+b) = f(a) + f(b) \in I$, so that $a + b \in f^{-1}(I)$. On the other hand if $a \in R$ and
$b \in f^{-1}(I)$, then $f(a) \in S$ and $f(b) \in I$. This means that $f(ab) = f(a) f(b) \in I$, whence
$ab \in f^{-1}(I)$. Thus $f^{-1}(I)$ is a partial ideal of $R$.

Now suppose that $I$ is prime. The fact that $I \neq S$ implies that $f^{-1}(I) \neq R$, thanks to
Lemma~\ref{improper partial ideal}. If $a \comm b$ in $R$ are such that $ab \in f^{-1}(I)$, then $f(a) \comm f(b)$
and $f(a) f(b) = f(ab) \in I$. Because $I$ is prime, either $f(a) \in I$ or $f(b) \in I$. In other words,
either $a \in f^{-1}(I)$ or $b \in f^{-1}(I)$. This proves that $f^{-1}(I)$ is prime.
\end{proof}

\begin{definition}
The rule assigning to each ring $R$ the set $\partSpec(R)$ of prime partial ideals of $R$, and to each ring homomorphism
$f \colon R \to S$ the map of sets
\begin{align*}
\partSpec(S) &\to \partSpec(R) \\
P &\mapsto f^{-1}(P),
\end{align*}
is a contravariant functor from the category of rings to the category of sets. We  denote this functor by
$\partSpec \colon \Ring \to \Set$, extending the notation introduced in Definition~\ref{prime partial ideal definition}.
\end{definition}

Notice immediately that the restriction of $\partSpec$ to $\CommRing$ is equal to $\Spec$, and therefore
the functor $\partSpec$ gives an object of the category $\restrict^{-1}(\Spec)$ defined earlier in this
section.
Of course, this functor could be defined on the category of all partial algebras and partial algebra homomorphisms.
But because our primary interest is in the category of rings, we have chosen to restrict our definition to that category.

\begin{example}
\label{domain example}
Recall that an ideal $P \lhd  R$ is \emph{completely prime} if $R/P$ is a domain; that is, $P \neq R$ and
for $a,b \in R$, $ab \in P$ implies that either $a \in P$ or $b \in P$. Certainly every completely prime ideal
of a ring is a prime partial ideal. Thus every domain has a prime partial ideal: its zero ideal.
Recalling Proposition~\ref{division ring partial ideals} we conclude that the zero ideal of a division ring
$D$ is its unique prime partial ideal, so that $\partSpec(D)$ is a singleton.
\end{example}

\separate

The universal property of $\partSpec$ will be established in Theorem~\ref{universal Spec theorem} below.
In preparation, we observe that a partial ideal of a ring is equivalent to a choice of ideal in every
commutative subring. 
For a partial algebra $R$ over a commutative ring $k$, we let $\C_k(R)$ denote the partially ordered set of
all commeasurable subalgebras of $R$.
(In case $R$ is a ring, $\C(R) := \C_{\Z}(R)$ is the poset of commutative subrings of $R$.)
Recall that a subset $\setS$ of a partially ordered set $X$ is \emph{cofinal} if for every $x \in X$
there exists $s \in \setS$ such that $x \leq s$.

\begin{lemma}
\label{data determining partial ideal}
Each of the following data uniquely determines a partial ideal of a partial algebra $R$:
\begin{enumerate}
\item A rule $I$ that associates to each commeasurable subalgebra $C \subseteq R$ an ideal $I(C) \lhd C$ such that,
if $C \subseteq C'$ are commeasurable subalgebras of $R$, then $I(C) = I(C') \cap C$;
\item A rule $I$ that associates to each commeasurable subalgebra $C \subseteq R$ an ideal $I(C) \lhd C$ such that,
if $C_1$ and $C_2$ are commeasurable subalgebras of $R$, then $I(C_1) \cap C_2 = C_1 \cap I(C_2)$;
\item For a cofinal set $\setS$ of commeasurable subalgebras of $R$, a rule $I$ that associates to each $C \in \setS$
an ideal $I(C) \lhd C$ such that, if $C_1$ and $C_2$ are in $\setS$, then $I(C_1) \cap C_2 = C_1 \cap I(C_2)$;
\item A rule $I$ that associates to each maximal commeasurable subalgebra $C \subseteq R$ an ideal $I(C) \lhd C$ such
that, if $C_1$ and $C_2$ are maximal commeasurable subalgebra of $R$, then $I(C_1) \cap C_2 = C_1 \cap I(C_2)$.
\end{enumerate}
\end{lemma}

\begin{proof}
First notice that the rules described in~(1) and~(2) are equivalent. For if $I$ satisfies~(1), then for any
$C_1,C_2 \in \C_k(R)$ we have
\begin{align*}
I(C_1) \cap C_2 &= I(C_1) \cap (C_1 \cap C_2) \\
&= I(C_1 \cap C_2) \\
&= I(C_2) \cap (C_1 \cap C_2) \\
&= I(C_2) \cap C_1.
\end{align*}
Thus $I$ satisfies~(2).
Conversely, if $I$ satisfies~(2) and if $C,C' \in \C(R)$ are such that $C \subseteq C'$, then
\[
I(C) = I(C) \cap C' = C \cap I(C'),
\]
proving that $I$ satisfies~(1).

The equivalence of the rules described in (2)--(4) is straightforward to verify. 
To complete the proof, we show that the data described in~(1) uniquely determines a partial ideal of $R$.
Given a rule $I$ as in~(1), the set $J = \bigcup_{C \in \C_k(R)} I(C) \subseteq R$ is certainly a partial ideal
of $R$. Conversely, given a partial ideal $J$ of $R$, the assignment $I$ sending $C \mapsto I(C) := J \cap C$
satisfies~(1). Clearly these maps $I \mapsto J$ and $J \mapsto I$ are mutually inverse.
\end{proof}

A choice of a prime ideal in each commutative subring of a ring $R$ can be viewed as an element of the
product $\prod_{C \in \C(R)} \Spec(C)$.
The above characterization~(1) of partial ideals says that the prime partial ideals can be identified
with those elements $(P_C)_{C \in \C(R)}$ of this product such that for every $C,C' \in \C(R)$
with $C \subseteq C'$, one has $P_{C'} \cap C = P_C$. This fact is used below.

The last step before the main result of this section is to give an alternative description
of $\partSpec$ as a certain limit.
We recall the ``product-equalizer'' construction of limits in the category of sets
(see~\cite[V.2]{MacLane}).
Let $D \colon J \to \Set$ be a diagram (i.e., $D$ is a functor and $J$ is a small category).
Then the limit of $D$ can be formed explicitly as
\[
\invlim_J D = \left\{ \left. (x_j) \in \prod_{j \in \Obj(J)}D(j) \, \right| \,
\text{$D(f)(x_i) = x_j$ for all $i,j \in \Obj(J)$ and all $f \colon i \to j$ in $J$} \right\},
\]
with the morphisms $\invlim D \to D(j)$ defined for each $j \in \Obj(J)$ via projection.

For a ring $R$, we view the partially ordered set $\C(R)$ defined above as a category
by considering each element of $\C(R)$ as an object and each inclusion as a morphism.
(The appropriate analogue of this category in the context of $C^*$-algebras
makes a key appearance in the definition of the Bohrification
functor~\cite[Def.~4]{HeunenLandsmanSpitters} of Heunen, Landsman, and Spitters.)
The functor that is shown to be isomorphic to $\partSpec$ in the following
proposition is very close to one defined by van den Berg and Heunen
in~\cite[Prop.~5]{BergHeunen1}.

\begin{proposition}
\label{isomorphic functors}
The contravariant functor $\partSpec \colon \Ring \to \Set$ is isomorphic to the functor
defined, for a given ring $R$, by
\[
R \mapsto \invlim_{C \in \C(R)^{\op}} \Spec(C).
\]
This isomorphism preserves the isomorphism of functors $\partSpec|_{\CommRing} \cong \Spec$.
\end{proposition}

\begin{proof}
For any ring $R$, we have the following isomorphisms of sets:
\begin{align*}
\invlim\nolimits_{C \in \C(R)^{\op}} \Spec(C) 
&= \left\{ \left. (P_C) \in \prod_{C \in \C(R)} \Spec(C) \, \right| \,
\begin{gathered}
\text{for all inclusions } i \colon C \hookrightarrow C',\\ \Spec(i)(P_{C'}) = P_{C}
\end{gathered}
\right\} \\
&= \left\{ \left. (P_C) \in \prod_{C \in \C(R)} \Spec(C) \, \right| \,
\begin{gathered}
\text{for all inclusions } C \subseteq C', \\ \ P_{C'} \cap C = P_{C} 
\end{gathered}
\right\} \\
&\cong \partSpec(R),
\end{align*}
where the last isomorphism comes from Lemma~\ref{data determining partial ideal} (and the discussion
that followed).
These isomorphisms are natural in $R$ and thus provide an isomorphism of functors.
\end{proof}

We will now show that $\partSpec$ is our desired ``universal $\Spec$ functor.''
In fact, we prove a stronger result stating that it is universal among all contravariant
functors $\Ring \to \Set$ whose restriction to $\CommRing$ has a natural transformation
to $\Spec$ that is not necessarily an isomorphism. This is made precise below.

Given functors $K \colon \cat{A} \to \cat{B}$ and $S \colon \cat{A} \to \cat{C}$,
we recall that the \emph{(right) Kan extension of $S$ along $K$} is a functor
$R \colon \cat{B} \to \cat{C}$ along with a natural transformation $\varepsilon \colon RK \to S$
such that for any other functor $F \colon \cat{B} \to \cat{C}$ with a natural transformation
$\eta \colon FK \to S$ there is a unique natural transformation $\delta \colon F \to R$
such that $\eta = \epsilon \circ (\delta K)$.
(The ``composite'' $\delta K \colon FK \to RK$ of a functor with a natural transformation
is a common shorthand for the \emph{horizontal composite} $\delta \circ \mathbf{1}_{K}$
of the identity natural transformation $\mathbf{1}_{K} \colon K \to K$ with $\delta$,
so that $\delta K(X) = \delta(K(X)) \colon FK \to RK$ for any $X \in \cat{A}$;
see~\cite[II.5]{MacLane} for information on horizontal composition.)
When $K \colon \cat{A} \to \cat{B}$ is an inclusion of a subcategory
$\cat{A} \subseteq \cat{B}$, notice that $FK = F|_{\cat{A}}$ is the restriction.
In this case, the natural transformation $\delta K \colon FK \to RK$ is the induced
natural transformation of the restricted functors
$\delta|_{\cat{A}} \colon F|_{\cat{A}} \to R|_{\cat{A}}$.

\begin{theorem}
\label{universal Spec theorem}
The functor $\partSpec \colon \Ringop \to \Set$, along with the identity natural
transformation $\partSpec|_{\CommRingop} \to \Spec$, is the Kan extension of the functor
$\Spec \colon \CommRingop \to \Set$ along the embedding $\CommRingop \subseteq \Ringop$.
In particular, $\partSpec$ is a terminal object in the category $\restrict^{-1}(\Spec)$.
\end{theorem}

\begin{proof}
Let $F \colon \Ring \to \Set$ be a contravariant functor with a fixed natural
transformation $\eta \colon F|_{\CommRing} \to \Spec$. We need to show that there
is a unique natural transformation $\delta \colon F \to \partSpec$ that induces
$\eta$ upon restriction to $\CommRing \subseteq \Ring$. To construct $\delta$,
fix a ring $R$. 
For every commutative subring $C$ of $R$, the inclusion $C \subseteq R$
gives a morphism of sets $F(R) \to F(C)$ and $\eta$ provides a morphism
$\eta_C \colon F(C) \to \Spec(C)$; these compose to give morphisms $F(R) \to \Spec(C)$.
By naturality of the morphisms involved, these maps out of $F(R)$ collectively
form a cone over the diagram obtained by applying $\Spec$ to the diagram
$\C(R)$ of commutative subrings of $R$.
By the universal property of the limit, there exists a unique dotted
arrow making the square below commute for all $C \in \C(R)$:
\[
\begin{tikzpicture}
  \path 
    (0,0) node (FR) {$F(R)$} +(3.5,0) node (limSpec) {$\invlim\limits_{C \in \C(R)} \Spec(C)$} + (7,0) node(pSpecR) {$\partSpec(R)$}
      +(0,-2) node (FC) {$F(C)$} +(3.5,-2) node (SpecC) {$\Spec(C)$.}
  ;
  \draw[dashed,->] (FR) -- (limSpec);
  \draw[->] (FR) -- (FC);
  \draw[->] (limSpec) -- (SpecC);  
  \draw[->] (FC) edge node[above]{$\eta_C$} (SpecC);
  \draw[->] (limSpec) edge node[above=-3pt]{$\sim$} (pSpecR);
  \draw[->] (FR) edge[bend left=20] node[above]{$\delta_R$} (pSpecR);
\end{tikzpicture}
\]
These morphisms $\delta_R$ form the components of a natural transformation
$\delta \colon F \to \partSpec$. By construction, $\delta$ induces $\eta$
when restricted to $\CommRing$. Uniqueness of $\delta$ is guaranteed by the
uniqueness of dotted arrow used to define $\delta_R$ above.

The second sentence of the theorem follows from the first by applying the universal
property of the Kan extension in the special case where $F \colon \Ringop \to \Set$
is a functor with a natural transformation $\eta \colon F|_{\CommRingop} \to \Spec$
that is an isomorphism.
\end{proof}

\section{Morphisms of partial algebras and the Kochen-Specker Theorem}
\label{Kochen-Specker section}

Having defined our universal functor $\partSpec$ extending $\Spec$, we must
now determine its value on the algebra $\M_n(\Complex)$.
The first result of this section establishes a connection between the
partial prime ideals of this algebra and certain morphisms of partial
algebras.

We recall a relevant fact from commutative algebra. Let $C$ be a finite dimensional
commutative algebra over an algebraically closed field $k$. Such an algebra
is artinian, so all of its prime ideals are maximal. Given a maximal ideal
$\mathfrak{m} \subseteq C$, the factor $k$-algebra $C/\mathfrak{m}$ is a finite dimensional field
extension of the algebraically closed field $k$ and thus is is isomorphic to $k$.
Hence $\Spec(C)$ is in bijection with the set of $k$-algebra homomorphisms $C \to k$.
This situation is generalized below.

\begin{proposition}
\label{finite dimensional partial Spec}
Let $R$ be partial algebra over an algebraically closed field $k$ such that every element of $R$
is algebraic over $k$ (e.g., $R$ is a finite dimensional $k$-algebra).
Then there is a bijection between the set $\partSpec(R)$ and the set of all morphisms of partial
$k$-algebras $f \colon R \to k$, which associates to each such morphism $f$ the inverse image
$f^{-1}(0)$.
\end{proposition}

\begin{proof}
Because $R$ consists of algebraic elements, every element of $R$ generates a finite dimensional
commeasurable subalgebra. In other words, $R$ is the union of its finite dimensional commeasurable
subalgebras.

Given a morphism $f \colon R \to k$ of partial $k$-algebras, the set $P_f := f^{-1}(0) \subseteq R$
is a prime partial ideal of $R$ according to Lemma~\ref{preimage lemma}. 
Furthermore, for each finite dimensional commeasurable subalgebra $C \subseteq R$, the prime ideal
$C \cap P_f \lhd C$ is maximal.  Thus the restriction $f|_C$ must be equal to the canonical
homomorphism $C \twoheadrightarrow C/(P_f \cap C) \overset{\sim}{\longrightarrow} k$. 

Conversely, suppose that $P \subseteq R$ is a prime partial ideal. We define a function $f \colon R \to k$
as follows. As before, for each finite dimensional commeasurable subalgebra $C \subseteq R$ the
prime ideal $P \cap C$ of $C$ is a maximal ideal. Thus we may define $g_C \colon C \to k$ via
the quotient map $C \twoheadrightarrow C/(P \cap C) \overset{\sim}{\longrightarrow} k$.
Notice that for finite dimensional commeasurable subalgebras $C \subseteq C'$, the following
diagram commutes:
\[
\xymatrix{
C \ar@{^{(}->}[d] \ar@{->>}[r] & C/(P \cap C) \ar[d] \ar[r]^-{\sim} & k \ar@{=}[d] \\
C' \ar@{->>}[r] & C'/(P \cap C') \ar[r]^-{\sim} & k.
}
\]
Thus there is a well-defined function $f_P \colon R \to k$ given, for any $r \in R$, by
$f_P(r) = g_C(r)$ for any finite dimensional commeasurable subalgebra $C$ of $R$ containing $r$
(such as $C = k[r] \subseteq R$).
It is clear from the construction of $f_P$ that $f_P^{-1}(0) = P$.

We have defined maps $P \mapsto f_P$ and $f \mapsto P_f$. The last sentences of the previous
two paragraphs show that these assignments are mutually inverse, completing the proof.
\end{proof}

\separate

Thus the proof of Theorem~\ref{main theorem} is reduced to understanding the morphisms of
partial $\Complex$-algebras $\M_n(\Complex) \to \Complex$.
The \emph{Kochen-Specker Theorem} provides just the information that we need.
This theorem, due to S.~Kochen and E.~Specker~\cite{KochenSpecker}, is a ``no-go theorem'' from quantum
mechanics that rules out the existence of certain types of hidden variable theories.
Probability is an inherent feature in the mathematical formulation of quantum physics; only the evolution
of the probability amplitude of a system is computed.
A hidden variable theory is, roughly speaking, a theory devised to explain quantum mechanics by
predicting outcomes of all measurements \emph{with certainty}.

The observable quantities of a quantum system are mathematically represented by self-adjoint operators in a
$C^*$-algebra. The Heisenberg Uncertainty Principle implies that if two such operators do not commute,
then the exact values of the corresponding observables cannot be simultaneously determined. On the other
hand, commuting observables have no uncertainty restriction imposed upon them by Heisenberg's principle.
In ~\cite{KochenSpecker} Kochen and Specker argued that a hidden variable theory should assign a real value
to each observable of a quantum system in such a way that values of the sum or product of commuting
observables is equal to the sum or product of their corresponding values. That is to say, Kochen and Specker's
notion of a hidden variable theory is a morphism of partial $\mathbb{R}$-algebras from the partial algebra
of observables to $\mathbb{R}$. With this motivation, Kochen and Specker showed that no such morphism
exists.

\begin{KochenSpeckerTheorem}
\label{Kochen-Specker theorem}
Let $n \geq 3$, and for $A := \M_n(\mathbb{C})$ let $A_{sa} \subseteq A$ denote the subset of
self-adjoint elements of $A$.
There does not exist a morphism of partial $\mathbb{R}$-algebras $f \colon A_{sa} \to \mathbb{R}$.
\end{KochenSpeckerTheorem}

Actually,~\cite{KochenSpecker} establishes this result for $n=3$, but it is often cited in the literature
for $n \geq 3$. Because the reduction to the case $n=3$ is straightforward, we include it below.

\begin{proof}[Proof for $n>3$]
We assume that the result holds for $n=3$, as proved in~\cite{KochenSpecker}. Let $n>3$, and assume
for contradiction that there is a morphism of partial algebras $f \colon A_{sa} \to \mathbb{R}$. Let
$P_i = E_{ii} \in A_{sa}$ be the orthogonal projection onto the $i$th basis vector. Then $\sum P_i = I$ and
$P_i P_j = \delta_{ij} P_i$. In particular, because $f$ is a morphism of partial algebras we have
$\sum f(P_i) = f(\sum P_i) = 1$.  Furthermore, each $f(P_i) = f(P_i^2) = f(P_i)^2$ must equal either~$0$
or~$1$. So the values $f(P_i)$ are all equal to $0$, except for one $P_j$ with $f(P_j) = 1$. 

Choose two of the other projections $P_i$ to get a set of three distinct projections $P_j$, $P_k$, and $P_{\ell}$.
Then $E := P_j + P_k + P_{\ell}$ is an orthogonal projection, so there is an isomorphism of the corner algebra
$EAE \cong \M_3(\mathbb{C})$ that preserves self-adjoint elements. Now the restriction of $f$ to
$(EAE)_{sa} = EAE \cap A_{sa}$ satisfies all properties of a morphism of partial $\mathbb{R}$-algebras,
except possibly the preservation of the multiplicative identity. But the multiplicative identity of $(EAE)_{sa}$
is $E$ and $f(E) = f(P_j) + f(P_k) + f(P_{\ell}) = 1$, proving that $f$ is a morphism of partial algebras. This
contradicts the Kochen-Specker Theorem in dimension~3.
\end{proof}

\smallseparate

In Corollary~\ref{Kochen-Specker corollary} below, we will establish an analogue of the Kochen-Specker
Theorem that is more suitable for our purposes.
First we require one preparatory result.  Given an element $x$ of a partial ring $R$, we will say that
another element $y \in R$ is \emph{an inverse of $x$} if $x \perp y$ and $xy = 1$.
(Such an element need not be unique! An example of an element with two inverses is easily
constructed by taking two copies of a Laurent polynomial ring $k[x_1, x_1^{-1}]$,
$k[x_2, x_2^{-1}]$, ``gluing'' them by identifying $k[x_1]$ with $k[x_2]$, and declaring
$x_i^{-1} \comm x_1 = x_2$ but with the $x_i^{-1}$ not commeasurable to one another.
An inverse $y$ of $x$ is unique if $y$ is commeasurable to all elements of $R$ that are commeasurable
to $x$. We thank George Bergman for these observations.)
The following argument is a standard one. It basically appeared
in~\cite[pp.81--82]{KochenSpecker}, and it even has roots in the theory of the Gelfand spectrum
of $C^*$-algebras.

\begin{lemma}
\label{eigenvalue lemma}
Let $R$ be a partial algebra over a commutative ring $k \neq 0$, and let $f \colon R \to k$ be a
morphism of partial $k$-algebras.
Then for any $r \in R$, the element $r-f(r) \in R$ does not have an inverse.
In particular, if $k$ is a field and $R = \M_n(k)$, then $f(r) \in k$ is an eigenvalue of $r$.
\end{lemma}

\begin{proof}
If $r-f(r)$ has an inverse $u \in R$, then $k=0$ by the following equation:
\begin{align*}
1 &= f(1) \\
&= f((r - f(r))u) \\
&= f(r-f(r)) f(u) \\
&= (f(r) - f(f(r) \cdot 1)) f(u) \\
&= (f(r) - f(r)) f(u) \\
&= 0. \qedhere
\end{align*}
\end{proof}

We now have the following reformulation of the Kochen-Specker Theorem that is more
appropriate to our needs. (One could think of it as a ``complex-valued,'' rather than
``real-valued,'' Kochen-Specker Theorem.)
Together with Proposition~\ref{finite dimensional partial Spec}, this constitutes the
final ``key result'' used in the proof of Theorem~\ref{main theorem}.

\begin{corollary}
\label{Kochen-Specker corollary}
For any $n \geq 3$, there is no morphism of partial $\mathbb{C}$-algebras $\M_n(\mathbb{C}) \to \mathbb{C}$.
\end{corollary}

\begin{proof}
Let $A = \M_n(\mathbb{C})$. Every self-adjoint matrix in $A$ has real eigenvalues, so
Lemma~\ref{eigenvalue lemma} implies that a morphism of partial $\Complex$-algebras
$A \to \Complex$ restricts to a morphism of partial $\mathbb{R}$-algebras $A_{sa} \to \mathbb{R}$.
But such morphisms are forbidden by the Kochen-Specker Theorem~\ref{Kochen-Specker theorem}.
\end{proof}

It is natural to ask what is the status of Corollary~\ref{Kochen-Specker corollary} in the case
$n=2$. Regarding their original theorem, Kochen and Specker demonstrated the existence of a
morphism of partial $\mathbb{R}$-algebras $\M_2(\mathbb{C})_{sa} \to \mathbb{R}$
in~\cite[\S6]{KochenSpecker}, showing that the Kochen-Specker Theorem does not extend to $n = 2$.
Similarly, Corollary~\ref{Kochen-Specker corollary} does not extend to $n = 2$. There exist
morphisms of partial algebras $\M_2(\Complex) \to \Complex$, and we can describe all of them
as follows.
Incidentally, this result also shows that the statement of Theorem~\ref{main theorem} is
not valid in the case $n = 2$; the functor $F = \partSpec$ assigns a nonempty set
(of cardinality $2^{2^{\aleph_0}}$, in fact!)\ to $\M_2(\Complex)$.

\begin{proposition}
\label{the case n = 2}
Let $k$ be an algebraically closed field, and let $\mathcal{I} \subseteq A := \M_2(k)$ be any set of
idempotents such that the set of all idempotents of $A$ is partitioned as
\[
\{ 0, 1\} \sqcup \mathcal{I} \sqcup \{ 1-e : e \in \mathcal{I} \}.
\]
Then for every function $\alpha \colon \mathcal{I} \to \{0,1\}$ there is a morphism of partial
$k$-algebras $f_{\alpha} \colon A \to k$ such that the restriction of $f$ to $\mathcal{I}$
is $\alpha \colon \mathcal{I} \to \{0,1\} \subseteq k$.
Moreover, there are bijective correspondences between:
\begin{itemize}
\item the set of functions $\alpha \colon \mathcal{I} \to \{0,1\}$;
\item the set of morphisms of partial $k$-algebras $A \to k$; and
\item the set of prime partial ideals of $A$;
\end{itemize}
given by $\alpha \leftrightarrow f_{\alpha} \leftrightarrow f_{\alpha}^{-1}(0)$.
\end{proposition}

\begin{proof}
First we construct a commutative $k$-algebra $B$ with a morphism of partial $k$-algebras $h \colon A \to B$.
Let $\mathcal{N}$ be a set of nonzero nilpotent elements of $A$ such that every nonzero nilpotent matrix in $A$
has exactly one scalar multiple in $\mathcal{N}$. Let $B$ be the commutative $k$-algebra
$B := k[ x_e, x_n : e \in \mathcal{I}, n \in \mathcal{N}]$ with relations $x_e^2 = x_e$ for $e \in \mathcal{I}$
and $x_n^2 = 0$ for $n \in \mathcal{N}$.

A result of Schur~\cite{Schur} (also proved more generally by Jacobson~\cite[Thm.~1]{Jacobson})
implies that every maximal commutative subalgebra of $A$ is has $k$-dimension~2. Thus the
intersection of two distinct commutative subalgebras of $A$ is the scalar subalgebra
$k \subseteq A$. This makes it easy to see that a function $h \colon A \to B$ is a morphism of
partial $k$-algebras if and only if its restriction to every 2-dimensional commutative subalgebra
of $A$ is a $k$-algebra homomorphism.

Now define a function $h \colon A \to B$ as follows. For each scalar $\lambda \in k \subseteq A$,
we set $h(\lambda) = \lambda \in k \subseteq B$. Now assume $a \in A \setminus k$. Then $k[a]$
is a 2-dimensional commutative subalgebra of $A$. Because the only 2-dimensional algebras over
the algebraically closed field $k$ are $k \times k$ and $k[\varepsilon]/(\varepsilon^2)$, there
exists $b \in \mathcal{I} \sqcup \mathcal{N}$ such that $k[a] = k[b]$. The careful choice of
the sets $\mathcal{I}$ and $\mathcal{N}$ ensures that this $b$ is unique.
Thus it suffices to define $h$ on each $k[b]$. But for $b \in \mathcal{I} \sqcup \mathcal{N}$, the
map $k[b] \to B$ defined by sending $b \mapsto x_b$ is clearly a homomorphism of $k$-algebras.
We define the restriction of $h$ to $k[a] = k[b]$ to be this homomorphism, which in particular
defines the value $h(a)$.

Certainly $h$ is well-defined, and it is a morphism of partial algebras because its restriction
to every 2-dimensional subalgebra is an algebra homomorphism.
Thus we have constructed a morphism of partial algebras $h \colon A \to B$.

Given a function $\alpha \colon \mathcal{I} \to \{0,1\}$, there exists a $k$-algebra homomorphism
$g_{\alpha} \colon B \to k$ given by sending $x_e \mapsto \alpha(e) \in k$ for $e \in \mathcal{I}$
and $x_n \mapsto 0$ for $n \in \mathcal{N}$. So the composite
$f_{\alpha} := g_{\alpha} \circ h$ is a morphism of partial $k$-algebras whose restriction to
$\mathcal{I}$ is equal to $\alpha$.
The bijection between the three sets in the statement of the proposition follows directly from
Proposition~\ref{finite dimensional partial Spec} above and Lemma~\ref{values on idempotents}
below.
\end{proof}

\begin{lemma}
\label{values on idempotents}
Let $R$ be a partial algebra over an algebraically closed field $k$ in which every element is
algebraic (e.g., $R$ is a finite dimensional $k$-algebra).
A morphism of partial algebras $R \to k$ is uniquely determined by its restriction to the set of
idempotents of $R$.
\end{lemma}

\begin{proof}
Let $f \colon R \to k$ be a morphism of partial $k$-algebras, and let $C \subseteq R$ be a
finite dimensional commeasurable subalgebra of $R$.
Because $R$ is the union of its finite dimensional commeasurable subalgebras, it is enough to show
that the restriction of $f$ to $C$, which is a $k$-algebra homomorphism $C \to k$, is uniquely
determined by its values on the idempotents of $C$.

Because $C$ is finite dimensional it is artinian and thus is a finite direct sum of local $k$-algebras.
Write $C = A_1 \oplus \cdots \oplus A_n$ where each $(A_i, M_i)$ is local and the identity element
of $A_i$ is $e_i$, an idempotent of $C$. Since $k$ is algebraically closed, each of the residue
fields $A_i/M_i$ is isomorphic to $k$ as a $k$-algebra.
Thus each $A_i = ke_i \oplus M_i$. Because $A_i$ is finite dimensional, its Jacobson radical $M_i$
is nilpotent and hence is in the kernel of $f|_C$. 
It now follows easily that the restriction of $f$ to $C$ is determined by the values $f(e_i)$.
\end{proof}

\section{Proof and consequences of the main result}
\label{proof section}

We are now prepared to prove Theorem~\ref{main theorem}, the main ring-theoretic result of the paper.

\begin{proof}[Proof of Theorem~\ref{main theorem}]
Fix $n \geq 3$ and let $A = \M_n(\mathbb{C}))$.
According to Theorem~\ref{universal Spec theorem} there exists a natural transformation
$F \to \partSpec$.
By Proposition~\ref{finite dimensional partial Spec}, $\partSpec(A)$ is in bijection
with the set of morphisms of partial $\Complex$-algebras $A \to \mathbb{C}$.
No such morphisms exist according to Corollary~\ref{Kochen-Specker corollary} of the Kochen-Specker
Theorem, so $\partSpec(A) = \varnothing$.
The existence of a function $F(A) \to \partSpec(A) = \varnothing$ now implies that
$F(A) = \varnothing$.
\end{proof}

\smallseparate

It seems appropriate to mention some partial positive results that contrast with
Theorem~\ref{main theorem}.
One might hope that restricting to certain well-behaved ring homomorphisms could allow the
functor $\Spec$ to be ``partially extended.''
In this vein, Procesi has shown~\cite[Lem.~2.2]{Procesi} that if $f \colon R \to S$ is a
ring homomorphism such that $S$ is generated over $f(R)$ by elements centralizing $f(R)$,
then for every prime ideal $Q \lhd S$ the inverse image $f^{-1}(Q)$ is again prime. 
Furthermore, he showed in~\cite[Thm.~3.3]{Procesi} that if $R$ is a Jacobson PI ring and
$f \colon R \to S$ is a ring homomorphism such that $S$ is generated by the image $f(R)$ and
finitely many elements that centralize $f(R)$,  then for every maximal ideal $M \lhd S$ the
inverse image $f^{-1}(M)$ is a maximal ideal of $R$. (Although he only stated these results
for injective $f$, they are easily seen to hold more generally.)

On the other hand, one may try to replace functions between prime spectra by ``multi-valued
functions,'' which may send a single element of one set to many elements of another set.
For instance, one might consider a functor that maps each homomorphism $R \to S$ of
noncommutative rings to a \emph{correspondence} $\Spec(S) \to \Spec(R)$, which
sends a single prime ideal of $S$ to some nonempty finite set of prime ideals of $R$.
This notion was introduced by Artin and Schelter in~\cite[\S4]{ArtinSchelter} and studied in
further detail by Letzter in~\cite{Letzter}. There is an appropriate notion of ``continuity'' of
a correspondence, and it is shown in~\cite[Cor.~2.3]{Letzter} (see
also~\cite[Prop.~4.6]{ArtinSchelter}) that if $f \colon R \to S$ is a ring homomorphism
and $S$ is a PI ring, then the associated correspondence is continuous.
However, there exist homomorphisms between noetherian rings whose correspondence
is not continuous~\cite[\S2.5]{Letzter}.

\separate

We now present a few corollaries of Theorem~\ref{main theorem}. The first is a straightforward
generalization of that theorem replacing $\M_n(\Complex)$ with $\M_n(R)$ where $R$ is
any ring containing a field isomorphic to $\Complex$.

\begin{corollary}
\label{matrix algebra corollary}
Let $F \colon \Ring \to \Set$ be a contravariant functor whose restriction to the full subcategory of commutative rings is isomorphic to
$\Spec$.  If $R$ is any ring with a homomorphism $\Complex \to R$, then $F(\M_n(R)) = \varnothing$ for $n \geq 3$.
\end{corollary}

\begin{proof}
The homomorphism $\mathbb{C} \to R$ induces a homomorphism $\M_n(\mathbb{C}) \to \M_n(R)$. Thus
we have a set map $F(\M_n(R)) \to F(\M_n(\mathbb{C}))$. If $n \geq 3$ then by Theorem~\ref{main theorem}
the latter set is empty; hence the former set must also be empty.
\end{proof}

In the corollary above, $R$ can be any complex algebra. But rings that contain $\Complex$ as a
non-central subring, such as the real quaternions, are also allowed. On the other hand, suppose
that $R$ is a complex algebra such that $R \cong \M_n(R)$ for some $n \geq 2$. It follows that
$R \cong \M_n(R) \cong \M_n(\M_n(R)) \cong \M_{n^2}(R)$, so we may assume that $n \geq 4$. Then
the corollary implies that for functors $F$ as above, $F(R) \cong F(\M_n(R)) = \varnothing$.
For instance, if $V$ is an infinite dimensional $\Complex$-vector space and $R$ is the algebra
of $\Complex$-linear endomorphisms of $V$, then the existence of a vector space isomorphism
$V \cong V^{\oplus n}$ (any $n \geq 2$) implies the existence of an algebra isomorphism
$R \cong \M_n(R)$.

An attempt to extend the ideas above suggests one possible algebraic generalization of the
Kochen-Specker Theorem.
Suppose that $\partSpec(\M_n(\Z)) = \varnothing$ for some integer $n \geq 3$. For any ring
$R$ the canonical ring homomorphism $\Z \to R$ induces a morphism $\M_n(\Z) \to \M_n(R)$.
Then one would have $\partSpec(\M_n(R)) = \varnothing$.
It would follow that any contravariant functor $F \colon \Ring \to \Set$ whose restriction to
$\CommRing$ is isomorphic to $\Spec$ must assign the empty set to $\M_n(R)$ for any ring $R$.
This highlights the importance of the following question.

\begin{question}
Do there exist integers $n \geq 3$ such that $\partSpec(\M_n(\Z)) = \varnothing$?
\end{question}

If $\partSpec(\M_n(\Z))$ were in fact empty for all sufficiently large values of $n$, then this
would be a sort of ``integer-valued'' Kochen-Specker Theorem.

\smallseparate

The next corollary of Theorem~\ref{main theorem} concerns certain functors sending rings to
commutative rings.
Consider the functor $\Ring \to \CommRing$ that sends each ring $R$ to its ``abelianization''
$R/[R,R]$. Rings whose abelianization is zero are easy to find, and this functor necessarily
destroys all information about these rings.
One could try to abstract this functor by considering any functor $\Ring \to \CommRing$ whose
restriction to $\CommRing$ is isomorphic to the identity functor. The following result says that
every such ``abstract abelianization functor'' necessarily destroys matrix algebras.

\begin{corollary}
\label{abelianization corollary}
Let $\alpha \colon \Ring \to \CommRing$ be a functor such that the restriction of $\alpha$ to
$\CommRing$ is isomorphic to the identity functor.
Then for any ring $R$ with a homomorphism $\Complex \to R$ and any $n \geq 3$, one has
$\alpha(\M_n(R)) = 0$. In particular, $\alpha$ is not faithful.
\end{corollary}

\begin{proof}
Because $\alpha$ restricts to the identity functor on $\CommRing$, the contravariant functor
$F := \Spec \circ \alpha \colon \Ring \to \Set$ satisfies $F|_{\CommRing} \cong \Spec$.
For $n \geq 3$, Corollary~\ref{matrix algebra corollary} implies that
$\Spec(\alpha(\M_n(R)) = F(\M_n(R)) = \varnothing$. Hence the commutative ring
$\alpha(\M_n(R))$ is zero. 

To see that $\alpha$ is not faithful, fix $n \geq 3$ and consider that $\alpha$ induces a function
\[
\Hom_{\Ring}(\M_n(\Complex), \M_n(\Complex)) \to
\Hom_{\CommRing}(\alpha(\M_n(\Complex)), \alpha(\M_n(\Complex)) = \Hom_{\CommRing}(0, 0).
\]
The latter set is a singleton, while the former set is not a singleton (because $\M_n(\Complex)$
has nontrivial inner automorphisms). So the function above is not injective, proving that $\alpha$
is not faithful.
\end{proof}

Interestingly, this result does not hold in the case $n = 2$; we thank George Bergman for
this observation. Let $\alpha \colon \Ring \to \CommRing$ be the functor sending each ring
to the colimit of the diagram of its commutative subrings. 
Certainly $\alpha|_{\CommRing}$ is isomorphic to the identity functor on $\CommRing$.
One can check that for an algebraically closed field $k$, the commutative ring $\alpha(\M_2(k))$
is isomorphic to the algebra $B$ constructed in the proof of Proposition~\ref{the case n = 2}; in
particular, $\alpha(\M_2(k)) \neq 0$.
(At the very least, it is not hard to verify from the universal property of the colimit that there
exists a homomorphism $\alpha(\M_2(k)) \to B$, confirming that $\alpha(\M_2(k)) \neq 0$.)
Furthermore, one can show that this functor is initial among all ``abstract abelianization
functors,'' but the details will not be presented here.

\smallseparate

The final corollary of Theorem~\ref{main theorem} to be presented in this section is a
rigorous proof that the rule that assigns to each (not necessarily commutative) ring $R$ the
set $\Spec(R)$ of prime ideals of $R$ is ``not functorial.''
(Recall that an ideal $P \lhd R$ is \emph{prime} if, for all ideals $I, J \lhd R$,
$IJ \subseteq P$ implies that either $I \subseteq P$ or $J \subseteq P$.)
The fact that this assignment ``is not a functor'' seems to be common wisdom. 
(Specific mention of this idea in the literature is not widespread, but
see~\cite[pp.~1 and~36]{vanOystaeyenVerschoren} or~\cite[\S1]{Letzter} for examples.)
It is easy to verify that this assignment is not a functor in the natural way; that is, if
$f \colon R \to S$ a ring homomorphism and $P \lhd S$ is prime, one can readily see that
the ideal $f^{-1}(P) \lhd R$ need not be prime.
However, we are unaware of any rigorous statement or proof in the literature of the
precise result below.

\begin{corollary}
\label{prime ideals not functorial}
There is no contravariant functor $F \colon \Ring \to \Set$ whose restriction to the full
subcategory $\CommRing$ is isomorphic to $\Spec$ and such that, for every ring $R$,
the set $F(R)$ is in bijection with the set of prime ideals of $R$.
\end{corollary}

\begin{proof}
Assume for contradiction that such $F$ exists.
Fix $n \geq 3$. Because the zero ideal of $\M_n(\Complex)$ is (its unique) prime, the
assumption on $F$ implies $F(\M_n(\Complex)) \neq \varnothing$, violating
Theorem~\ref{main theorem}.
\end{proof}

This corollary can also be derived from an elementary argument that avoids using
Theorem~1.1. In fact, the statement can even be strengthened as follows.

\begin{proposition}
\label{Morita invariant Spec}
There is no contravariant functor $F\colon\Ring\to\Set$ whose restriction to $\CommRing$ is
isomorphic to $\Spec$ and such that $F$ satisfies either of the following conditions:
\begin{enumerate}
\item For some field $k$ and some integer $n \geq 2$, the set $F(\M_n(k))$ is a singleton;
\item $F$ is Morita invariant in the following sense: for any Morita equivalent rings $R$ and $S$, one has $F(R)\cong F(S)$.
\end{enumerate}
\end{proposition}

\begin{proof}
First notice that if $F$ satisfies condition~(2) above, then it satisfies condition~(1) because
$\M_n(k)$ is Morita equivalent to $k$, which would mean that $F(\M_n(k)) \cong F(k) \cong \Spec(k)$
is a singleton.
So assume for contradiction that there exists a functor $F$ as above satisfying~(1).

Fix $k$ and $n$ as in condition~(1).
Define $\pi := (1\ 2\ \cdots \ n) \in S_n$, a permutation of the set $\{1,2, \dots, n\}$.
Let $\rho$ be the automorphism of $k^n$ given by $(a_i) \mapsto (a_{\pi(i)})$,
let $P \in \M_n(k)$ be the permutation matrix whose $i$th row is the $\pi(i)$th
standard basis row vector, and let $\sigma$ be the inner automorphism of $\M_n(k)$
given by $\sigma(A) = PAP^{-1}$.
For the final piece of notation, let $\iota \colon k^n \hookrightarrow \M_n(k)$ be the
diagonal embedding.

The following equality of algebra homomorphisms $k^n \to \M_n(k)$ is elementary:
\[
\iota \circ \rho = \sigma \circ \iota.
\]
Applying the contravariant functor $F$ to this equation gives
$F(\rho) \circ F(\iota) = F(\iota) \circ F(\sigma)$.
By hypothesis the set $F(\M_n(k))$ is a singleton. Hence the automorphism $F(\sigma)$
of $F(\M_n(k))$ is the identity. It follows that
\begin{equation}
\label{trouble equation}
F(\rho) \circ F(\iota) = F(\iota).
\end{equation}
On the other hand $F(k^n) \cong \Spec(k^n) = \{1, \dots, n\}$, and under this
isomorphism $F(\rho)$ acts as $\Spec(\rho) = \pi^{-1}$ which has no fixed points.
Thus the image of the unique element of $F(\M_n(k))$ under $F(\iota)$ is distinct
from its image under $F(\rho) \circ F(\iota)$, contradicting~\eqref{trouble equation}
above.
\end{proof}

Because the set of prime ideals of a noncommutative ring is Morita invariant (for instance,
see~\cite[(18.45)]{Lectures}) the proposition above implies Corollary~\ref{prime ideals not functorial}.
Notice that Proposition~\ref{Morita invariant Spec} with $k = \Complex$ and $n = 2$ cannot
be derived from Theorem~\ref{main theorem} because that theorem does not apply to the
algebra $\M_2(\Complex)$, as explicitly shown in Proposition~\ref{the case n = 2}.

Of course, there are many important examples of invariants of rings extending $\Spec$ of a
commutative ring that respect Morita equivalence, aside from the set of prime two-sided ideals of a
ring.
Two examples are the prime torsion theories introduced by O.~Goldman in~\cite{Goldman} and 
the spectrum of an abelian category defined by A.~Rosenberg in~\cite{Rosenberg2}.
(Incidentally, both of these spectra arise from the theory of noncommutative localization.)
Each of these invariants is certainly useful in the study of noncommutative algebra, and they
have appeared in different approaches to noncommutative algebraic geometry.
Thus we emphasize that Proposition~\ref{Morita invariant Spec} does not in any way suggest
that such invariants should be avoided.
It simply reveals that we cannot hope for such invariants to be functors to $\Set$.

\section{The analogous result for $C^*$-algebras}
\label{C* section}

In this section we will prove Theorem~\ref{main C* theorem}, which obstructs extensions
of the Gelfand spectrum functor to noncommutative $C^*$-algebras.
We begin by reviewing some facts and setting some conventions about the category of $C^*$-algebras.
(Many of these basics can be found in~\cite[\S I.5]{Davidson} and~\cite[\S 4.1]{KadisonRingrose}.)
We emphasize that \emph{all $C^*$-algebras considered in this section are assumed to be unital}.
Let $\Cstar$ denote the category whose objects are unital $C^*$-algebras and whose
morphisms are identity-preserving $*$-homomorphisms. Such morphisms do not increase the norm
and are norm-continuous.
The only topology on a $C^*$-algebra to which we will refer is the norm topology.
A closed ideal of a $C^*$-algebra is always $*$-invariant, and the resulting factor algebra is a
$C^*$-algebra.
A \emph{$C^*$-subalgebra} of a $C^*$-algebra $A$ is a closed subalgebra $C \subseteq A$
that is invariant under the involution of $A$; such a subalgebra inherits the structure of a
Banach algebra with involution from $A$ and is itself a $C^*$-algebra with respect to this
inherited structure.
If $f \colon A \to B$ is a $*$-homomorphism, then the image $f(A) \subseteq B$ is always a
$C^*$-subalgebra.
The full subcategory of $\Cstar$ consisting of commutative unital $C^*$-algebras is denoted
by $\CommCstar$.
Finally, the reader may wish to see Section~\ref{introduction section} for the definition
of the (contravariant) Gelfand spectrum functor $\Gelf \colon \CommCstar \to \Set$.

\smallseparate

As in the ring-theoretic case, we define an appropriate category of functors in which we seek
a universal functor.
The inclusion of categories $\CommCstar \hookrightarrow \Cstar$ induces a \emph{restriction}
functor between functor categories
\begin{align*}
\restrict \colon \Fun(\Cstar^{\op}, \Set) &\to \Fun(\CommCstar^{\op}, \Set) \\
F &\mapsto F|_{\CommCstar^{\op}}.
\end{align*}
Again we define the ``fiber category'' over $\Gelf \in \Fun(\CommCstar^{\op}, \Set)$ to be the
category $\restrict^{-1}(\Gelf)$ of pairs $(F, \phi)$ where $F \in \Fun(\Cstar^{\op}, \Set)$ and
$\phi \colon \restrict(F) \overset{\sim}{\longrightarrow} \Gelf$ is an isomorphism of functors; a
morphism $\psi \colon (F, \phi) \to (F', \phi')$ in $\restrict^{-1}(\Gelf)$ is a natural transformation
$\psi \colon F \to F'$ such that $\phi' \circ \restrict(\psi) = \phi$.
Our first goal is to locate a final object of the category $\restrict^{-1}(\Gelf)$, which we view as
a ``universal noncommutative Gelfand spectrum functor.''

\separate

First we define the analogue of the spectrum $\partSpec$ of prime partial ideals.
To do so, it will be useful to think in terms of partial $C^*$-algebras, as defined
by van den Berg and Heunen in~\cite[\S 4]{BergHeunen1}.

\begin{definition}
A partial $C^*$-algebra $P$ is a partial $\Complex$-algebra with an involution
$* \colon P \to P$ and a function $\| \cdot \| \colon P \to \R$ such that any set
$S \subseteq P$ of pairwise commeasurable elements is contained in a set $T \subseteq P$
such that the restricted operations of $P$ endows $T$ with the structure of a
commutative $C^*$-algebra.
Such a subset $T \subseteq P$ is called a \emph{commeasurable $C^*$-subalgebra}
of $P$.
A \emph{$*$-morphism} of partial $C^*$-algebras $f \colon P \to Q$ is a morphism
of partial $\Complex$-algebras satisfying $f(a^*) = f(a)^*$ for all $a \in P$.
\end{definition}

It is very important to note that, unlike the ring-theoretic case, a $C^*$-algebra
with the commeasurability relation of commutativity is generally \emph{not} a partial
$C^*$-algebra.
What is true is that for any $C^*$-algebra $A$, the set $N(A) = \{a \in A : aa^* = a^* a\}$
of normal elements (with commutativity as commeasurability) is always a partial
$C^*$-algebra.
This makes use of the fact that any normal element of a $C^*$-algebra has the
property that its centralizer is a $*$-subalgebra; see~\cite{Fuglede}.
The assignment $A \mapsto N(A)$ defines a functor from the category of $C^*$-algebras
to the category of partial $C^*$-algebras, defined on a morphism $f \colon A \to B$
by restricting and corestricting $f$ to $N(f) \colon N(A) \to N(B)$;
see~\cite[Prop.~3]{BergHeunen1}.
(Because of this subtlety regarding normal elements, we will typically use $P,Q$ to
denote partial $C^*$-algebras and $A,B$ to denote full $C^*$-algebras.)

\begin{definition}
A \emph{partial closed ideal} of a partial $C^*$-algebra $P$ is a subset $I \subseteq P$
such that, for every commeasurable $C^*$-subalgebra $C \subseteq P$, the intersection
$I \cap C$ is a closed ideal of $C$. If, for every commeasurable $C^*$-subalgebra $C$
one has that $I \cap C$ is a maximal ideal of $C$, then $I$ is a
\emph{partial maximal ideal} of $N$.

Let $A$ be a $C^*$-algebra. We say that a subset $I \subseteq A$ is a \emph{partial
closed (resp.\ maximal) ideal} of $A$ if $I \subseteq N(A)$ and $I$ is a partial closed
(resp.\ maximal) ideal of the partial $C^*$-algebra $N(A)$.
\end{definition}

Because we require a partial closed ideal $I$ of a $C^*$-algebra $A$ to consist of
normal elements, $I$ is completely determined by its intersection with all commutative
subalgebras in the sense that $I = \bigcup_C (I \cap C)$, where $C$ ranges over all
commutative $C^*$-subalgebras of $A$.
This is true because an element of $A$ is normal if and only if it is contained in a
commutative $C^*$-subalgebra of $A$.

As in the ring-theoretic case, partial closed ideals behave well under homomorphisms.

\begin{lemma}
Let $f \colon P \to Q$ be a $*$-homomorphism of $C^*$-algebras, and let $I$ be a
partial closed (resp.\ maximal) ideal of $Q$.
The set $f^{-1}(I) \subseteq P$ is a partial closed (resp.\ maximal) partial ideal of $P$.

In particular, if $f \colon A \to B$ is a $*$-homomorphism between $C^*$-algebras and
$I \subseteq N(B) \subseteq B$ is a partial closed (resp.\ maximal) ideal, then
$f^{-1}(I) \cap N(A) \subseteq A$ is a partial closed (resp.\ maximal) ideal of $A$.
\end{lemma}

\begin{proof}
Let $C \subseteq P$ be a commeasurable $C^*$-subalgebra. We wish to show that
$f^{-1}(I) \cap C$ is a closed (resp.\ maximal) ideal of $C$. First notice that since $C$
consists of pairwise commeasurable elements, so does $f(C) \subseteq Q$.
Thus there is a commeasurable $C^*$-subalgebra $D \subseteq Q$ such that $f(C) \subseteq D$.
But then since $f$ (co)restricts to a $*$-homomorphism of (full) $C^*$-algebras $C \to D$,
its image $f(C) \subseteq D$ is a $C^*$-subalgebra. It follows that $f$ (co)restricts to
a $*$-homomorphism of (full, commutative) $C^*$-algebras $f|_C \colon C \to f(C)$.

Now $I \cap f(C)$ is a closed (resp.\ maximal) ideal in $f(C)$ by hypothesis, so its preimage
under $f|_C$ is a closed (resp.\ maximal) ideal of $C$. On the other hand,
$(f|_C)^{-1}(I \cap f(C))$ is easily seen to be equal to $f^{-1}(I) \cap C$. Hence the latter
is a closed (resp.\ maximal) ideal, as desired.
\end{proof}

This allows us to define a ``partial Gelfand spectrum'' functor.

\begin{definition}
We define a contravariant functor $\partGelf \colon \Cstar \to \Set$ by assigning to
every $C^*$-algebra $A$ the set $\partGelf(A)$ of partial maximal ideals of $A$, and
by assigning to each morphism $f \colon A \to B$ in $\Cstar$ the function
\begin{align*}
\partGelf(B) &\to \partGelf(A) \\
M &\mapsto f^{-1}(M) \cap N(A).
\end{align*}
(The only potential hindrance to functoriality is the preservation of composition of
morphisms, but this is easily verified. Alternatively, this is seen to be a functor
because it is the composite of the functor from $C^*$-algebras to partial $C^*$-algebras
$A \mapsto N(A)$ with the contravariant functor from partial $C^*$-algebras to sets
that sends a partial algebra to the set of its partial maximal ideals.)
\end{definition}

Notice that the restriction of $\partGelf$ to $\CommCstar$ is equal to the Gelfand
spectrum functor $\Gelf$, so that $\partGelf$ is an object of the category
$\restrict^{-1}(\Gelf)$.

As in Proposition~\ref{isomorphic functors}, the functor $\partGelf$ can be recovered
through a limit construction. For a $C^*$-algebra $A$, we let $\C^*(A)$ denote the
partially ordered set of its commutative $C^*$-subalgebras, viewed as a category in the
usual way.

\begin{proposition}
\label{isomorphic C* functors}
The contravariant functor $\partGelf \colon \Cstar \to \Set$ is isomorphic to the functor
defined, for a given $C^*$-algebra $A$, by
\[
A \mapsto \invlim_{C \in \C^*(A)^{\op}} \Gelf(C).
\]
This isomorphism preserves the isomorphism of functors $\partGelf|_{\CommCstar} = \Spec$.
\end{proposition}

We will not include a proof of this fact, but we will mention the main subtlety. The
appropriate analogue of Lemma~\ref{data determining partial ideal} (replacing each
occurrence of ``commeasurable subalgebra'' with ``commeasurable $C^*$-subalgebra'') still
holds, and is used as before to prove the present result. Here it is crucial to recall that a
partial closed ideal $I \subseteq A$ consists of normal elements, for this ensures that $I$
is determined by its intersection with all commutative $C^*$-subalgebras of $A$.

Just as before, this allows one to show that $\partGelf$ is a ``universal Gelfand
spectrum functor.''

\begin{theorem}
\label{universal Gelf theorem}
The functor $\partGelf \colon \Cstarop \to \Set$, with the identity natural transformation
$\partGelf|_{\CommCstar} \to \Gelf$, is the Kan extension of the functor
$\Gelf \colon \CommCstarop \to \Set$ along the embedding $\CommCstarop \subseteq \Cstarop$.
In particular, $\partGelf$ is a terminal object in the category $\restrict^{-1}(\Gelf)$.
\end{theorem}

\separate

Our next goal is to connect the functor $\partGelf$ to the Kochen-Specker Theorem,
in a manner similar to that of Section~\ref{Kochen-Specker section}.
We have the following analogue of Proposition~\ref{finite dimensional partial Spec}.
Its proof is completely analogous, and thus is omitted.

\begin{proposition}
\label{partial Gelf and morphisms}
Let $P$ be a partial $C^*$-algebra.
There is a bijection between the set of partial maximal ideals of $P$ and the set of
$*$-morphisms of partial $C^*$-algebras $f \colon P \to \Complex$, which associates to each
such morphism $f$ the inverse image $f^{-1}(0)$.

In particular, if $A$ is a $C^*$-algebra then there is a bijection between $\partGelf(A)$
and the set of $*$-morphisms of partial $C^*$-algebras $f \colon N(A) \to \Complex$.
\end{proposition}

We have effectively reduced the proof of Theorem~\ref{main C* theorem} to a question of the
existence of morphisms of partial $C^*$-algebras. Thus we are in a position to apply the
Kochen-Specker Theorem. The following corollary to Kochen-Specker is proved just like
Corollary~\ref{Kochen-Specker corollary}, relying upon Lemma~\ref{eigenvalue lemma}.

\begin{corollary}[A corollary of the Kochen-Specker Theorem]
\label{Kochen-Specker C* corollary}
For any $n \geq 3$, there is no $*$-morphism of partial $C^*$-algebras
$N(\M_n(\Complex)) \to \Complex$.
\end{corollary}

\separate

We are now ready to give a proof of Theorem~\ref{main C* theorem}, obstructing
extensions of the Gelfand spectrum functor.

\begin{proof}[Proof of Theorem~\ref{main C* theorem}]
Fix $n \geq 3$ and let $A = \M_n(\Complex) \in \Cstar$. 
By Theorem~\ref{universal Gelf theorem} there is a natural transformation
$F \to \partGelf$.
The set $\partGelf(A)$ is in bijection with the set of $*$-morphisms of partial
$C^*$-algebras $N(A) \to \Complex$ according to Proposition~\ref{partial Gelf and morphisms}.
By Corollary~\ref{Kochen-Specker C* corollary} of the Kochen-Specker Theorem there
are no such $*$-morphisms. Thus $\partGelf(A) = \varnothing$, and the existence
of a function $F(A) \to \partGelf(A)$ gives $F(A) = \varnothing$.
\end{proof}

\separate

The corollaries to Theorem~\ref{main theorem} given in Section~\ref{proof section} all have
analogues in the setting of $C^*$-algebras. For the most part we will omit the proofs of these
results because they are such straightforward adaptations of those given in
Section~\ref{proof section}.
First we provide an analogue of Corollary~\ref{matrix algebra corollary}, and we include its
proof only to illustrate how our restriction to unital $C^*$-algebras comes into play.

\begin{corollary}
Let $F \colon \Cstar \to \Set$ be a contravariant functor whose restriction to the full subcategory
of commutative $C^*$-algebras is isomorphic to $\Gelf$. Then for any $C^*$-algebra $A$ and
integer $n \geq3$, one has $F(\M_n(A)) = \varnothing$.
\end{corollary}

\begin{proof}
Because $A$ is unital, there is a canonical morphism of $C^*$-algebras $\Complex \to \Complex \cdot 1_A \subseteq A$.
This induces a $*$-morphism $\M_n(\Complex) \to \M_n(A)$. Thus there is a function
of sets $F(\M_n(A)) \to F(\M_n(\Complex))$, and the latter set is empty by Theorem~\ref{main C* theorem}.
Hence $F(\M_n(A)) = \varnothing$.
\end{proof}

As in the discussion following Corollary~\ref{matrix algebra corollary}, this result shows that if $A$ is a
unital $C^*$-algebra for which there is an isomorphism $A \cong \M_n(A)$ for some $n \geq 2$,
then for any functor $F$ as above, $F(A) = \varnothing$. As an example, we may take $A$ to
be the algebra of bounded operators on an infinite-dimensional Hilbert space.

Next is the appropriate analogue of Corollary~\ref{abelianization corollary}.

\begin{corollary}
Let $\alpha \colon \Cstar \to \CommCstar$ be a functor whose restriction to $\CommCstar$
is isomorphic to the identity functor. Then for every $C^*$-algebra $A$ and every $n \geq 3$,
one has $\alpha(\M_n(A)) = 0$
\end{corollary}

Finally, there is the following analogue of Corollary~\ref{prime ideals not functorial}.

\begin{corollary}
There is no contravariant functor $F \colon \Cstar \to \Set$ whose restriction to the
full subcategory $\CommCstar$ is isomorphic to $\Gelf$ and such that, for every $C^*$-algebra
$A$, the set $F(R)$ is in bijection with the set of primitive ideals of $A$.
\end{corollary}

This corollary can be obtained as a consequence either of Theorem~\ref{main C* theorem} or
of the obvious analogue of Proposition~\ref{Morita invariant Spec}. In fact, the
proof of the latter proposition (with $k = \Complex$) extends directly to the setting
of $C^*$-algebras because all of the homomorphisms used in its proof are actually
$*$-homomorphisms.

\section*{Acknowledgments}

I am grateful to Andre Kornell and Matthew Satriano for several stimulating conversations in
the early stages of this work, George Bergman for many insightful comments, Lance Small
for helpful references to the literature, and Theo Johnson-Freyd for advice on creating
Ti\emph{k}z diagrams.
Finally, I thank the referee for a number of suggestions and corrections that improved the
readability of the paper.

\bibliographystyle{amsplain}
\bibliography{EmptySpec.v3}
\end{document}